\documentclass[12pt]{amsart}

\usepackage[utf8]{inputenc}

\usepackage{amssymb}
\usepackage{hyperref}
\usepackage[capitalize]{cleveref}

\numberwithin{equation}{section}

\swapnumbers

\newtheorem{cor}[equation]{Corollary}
\newtheorem{lem}[equation]{Lemma}
\newtheorem{prop}[equation]{Proposition}
\newtheorem{thm}[equation]{Theorem}

\crefformat{lem}{Lemma~#2#1#3}
\Crefname{lem}{Lemma}{Lemmas}
\crefname{lem}{lemma}{lemmas}
\crefformat{cor}{Corollary~#2#1#3}
\Crefname{cor}{Corollary}{Corollaries}
\crefname{cor}{corollary}{corollaries}
\crefformat{prop}{Proposition~#2#1#3}
\Crefname{prop}{Proposition}{Propositions}
\crefname{prop}{proposition}{propositions}
\crefformat{thm}{Theorem~#2#1#3}
\Crefname{thm}{Theorem}{Theorems}
\crefname{thm}{theorem}{theorems}

\theoremstyle{definition}
\newtheorem{assumps}[equation]{Assumptions}
\newtheorem{defn}[equation]{Definition}
\newtheorem{facts}[equation]{Facts}
\newtheorem{eg}[equation]{Example}

\theoremstyle{remark}
\newtheorem*{ack}{Acknowledgments}
\newtheorem{rem}[equation]{Remark}
\newtheorem{rems}[equation]{Remarks}

\crefformat{rems}{Remark~#2#1#3}

\newcommand{\cartprod}{\mathbin{\raise0.7pt\hbox{\smaller[2]$\square$}}}
\newcommand{\iso}{\cong}
\newcommand{\skel}{\mathcal{S}}

\newcommand{\ZZ}{\mathbb{Z}}

\DeclareMathOperator{\Aut}{Aut}
\DeclareMathOperator{\Cay}{Cay}
\DeclareMathOperator{\colour}{\mathsf{colour}}

\newcommand{\pref}[1]{(\ref{#1})}
\newcommand{\fullref}[2]{\ref{#1}\pref{#1-#2}}
\newcommand{\fullcref}[2]{\cref{#1}\pref{#1-#2}}

\makeatletter
\newcommand{\noprelistbreak}{\smallskip\@nobreaktrue\nopagebreak} 
\makeatother

\makeatletter
\def\swappedhead#1#2#3{%
  \thmnumber{\@upn{(\@secnumfont#2)\@ifnotempty{#1}{\hskip0.4em}}}%
  \thmname{#1}%
  \thmnote{ {\the\thm@notefont(#3)}}}
\makeatother
\begin{document}

\title[Stability of Cayley graphs]{Stability of Cayley graphs on 
\\ abelian groups of odd order}

\author{Dave Witte Morris}
\address{\hskip-\parindent Department of Mathematics and Computer Science,
	University of Lethbridge, \newline 
	4401 University Drive,
	Lethbridge, Alberta, T1K~3M4, Canada}
\email{\href{mailto:Dave.Morris@uleth.ca}{Dave.Morris@uleth.ca}} 
\urladdr{\url{http://people.uleth.ca/~dave.morris/}}

\begin{abstract}
Let $X$ be a connected Cayley graph on an abelian group of odd order, such that no two distinct vertices of~$X$ have exactly the same neighbours. We show that the direct product $X \times K_2$ (also called the \emph{canonical double cover} of~$X$) has only the obvious automorphisms (namely, the ones that come from automorphisms of its factors $X$ and~$K_2$). This means that $X$ is 
``stable\rlap.\spacefactor=3000'' % \spacefactor=3000 tells TeX that this period is the end of a sentence
The proof is short and elementary.
The theory of direct products implies that $K_2$ can be replaced with members of a much more general family of connected graphs.
\end{abstract}

\date{\today}

\maketitle

\section{Introduction} \label{IntroSect}

The \emph{canonical bipartite double cover} \cite{CanCover} of a graph~$X$ is the bipartite graph~$BX$ with $V(BX) = V(X) \times \{0,1\}$, where
	\[ \text{$(v,0)$ is adjacent to $(w,1)$ in $BX$}
	\quad \iff \quad
	\text{$v$ is adjacent to~$w$ in~$X$} . \]
Letting $S_2$ be the symmetric group on the $2$-element set~$\{0,1\}$, it is clear that $\Aut X \times S_2$ is a subgroup of $\Aut BX$. If this subgroup happens to be all of $\Aut BX$, then it is easy to see (and well known) that $X$ must be connected, and must also be ``twin-free'' (see \cref{TwinFreeDefn} below).

B.\,Fernandez and A.\,Hujdurović \cite{FernandezHujdurovic} recently established that the converse is true when $X$ is a circulant graph of odd order. This had been conjectured by Y.-L.\,Qin, B.\,Xia, and S.\,Zhou \cite[Conj.~1.3]{QinXiaZhou}, who proved the special case where $X$ has prime order.  See the introductions of \cite{FernandezHujdurovic} and~\cite{QinXiaZhou} for additional history and motivation.

Circulant graphs are examples of ``Cayley graphs'' (see \cref{CayleyDefn} below), and both sets of authors asked whether the converse can be generalized to all Cayley graphs on abelian groups of odd order (\cite[Problem~3.3]{FernandezHujdurovic} and \cite[p.~157]{QinXiaZhou}). This note provides a short, elementary proof that the desired generalization is indeed true:

\begin{thm} \label{CayleyByK2}
If $X$ is a twin-free, connected Cayley graph on a finite abelian group of odd order, then $\Aut BX = \Aut X \times S_2$.
\end{thm}

\begin{rem}
A graph~$X$ is said to be \emph{stable} if $\Aut BX = \Aut X \times S_2$ \cite[p.~155]{QinXiaZhou}, so the \lcnamecref{CayleyByK2} can be rephrased as the statement that:
	\begin{quotation} \it
	Every twin-free, connected Cayley graph on a finite abelian group of odd order is stable.
	\end{quotation}
(However, the term ``stable graph'' is ambiguous, because it also has other meanings in graph theory \cite{BeckenbachBorndorfer-stable,GoddardWinter-stable}.)
\end{rem}

\begin{eg}[Hujdurović-Mitrović {[personal communication]}] \label{nonabelEg}
The word ``abelian'' cannot be deleted from the statement of the \lcnamecref{CayleyByK2}. For example, if 
 	\[ G = \langle\, a, x \mid a^3 = x^7 = 1, \, a^{-1} x a = x^2 \,\rangle  \]
is the nonabelian group of order~$21$, then $X = \Cay \bigl( G ; \{ a^{\pm1}, x^{\pm1}, (ax)^{\pm1} \} \bigr)$ is twin-free and connected, but \textsf{MAGMA} computer calculations show that $|{\Aut X}| = 42$ and $|{\Aut BX}| = 252$. % DWM verified both calculations with sagemath (2020-10-9)
\end{eg}

\begin{rem}[{}{\cite[Rem.~1.3]{FernandezHujdurovic}}]
For any Cayley graph~$X$ on an abelian group of odd order, the theorem makes it possible to obtain the automorphism group of~$BX$ from the automorphism group of~$X$. For simplicity, let us assume that $X$ is loopless. Then there exist integers $c,d \ge 1$, and a twin-free, connected, abelian Cayley graph~$Y$ of odd order, such that $X \iso \overline{K_c} \wr (Y \wr \overline{K_d})$. Then, by a well-known theorem of Sabidussi \cite{Sabidussi} on the automorphism group of a wreath product of graphs, the theorem implies that $\Aut BX = S_c \wr \bigl( (\Aut Y \times \ZZ_2) \wr S_d \bigr)$.
\end{rem}

The canonical bipartite double cover~$BX$ can be realized as the \emph{direct product} $X \times K_2$ (see \cref{DirectProdDefn}), and the theory of direct products (see \cref{DirectProdSect}) implies that the \lcnamecref{CayleyByK2} can be generalized by replacing $BX = X \times K_2$ with $X \times Y$, where $Y$ is any graph in a much more general family:

\begin{cor} \label{CayleyByY}
Let $X$ be a twin-free, connected Cayley graph on a finite abelian group of odd order, and let $Y$ be any twin-free, connected graph, such that either:
	\begin{enumerate} \itemsep=\smallskipamount 
	\item $Y$ is not bipartite, and\/ $|V(Y)|$ is relatively prime to\/ $|V(X)|$, 
	or
	\item $Y$ is bipartite, with bipartition $V(Y) = Y_0 \cup Y_1$, such that
		\noprelistbreak
		\begin{enumerate} \itemsep = \smallskipamount
		\item $|Y_0|$ and\/ $|Y_1|$ are relatively prime to\/ $|V(X)|$,
			and
		\item either
			$|Y_0| \neq |Y_1|$,
			or
		 	$Y$ has an automorphism that interchanges $Y_0$ and~$Y_1$.
		\end{enumerate}

	\end{enumerate}
Also assume that neither $X$ nor~$Y$ is the one-vertex trivial graph. Then 
	\[ \Aut (X \times Y) = \Aut X \times \Aut Y . \]
\end{cor}

\begin{rem}
In \cref{CayleyByK2}, the Cayley graph~$X$ can be allowed to have an edge-colouring that is invariant under translation by elements of~$G$ (see \fullcref{NotSimple}{colour}). However, the proof of \cref{CayleyByY} does not allow  colours on the edges.
\end{rem}

Here is an outline of the paper:
\noprelistbreak
	\begin{enumerate} \itemsep=\smallskipamount 
	\item[\S\ref{IntroSect}.] Introduction
	\item[\S\ref{LemmaSect}.] A crucial lemma
	\item[\S\ref{CommentsSect}.] Comments on the lemma (optional)
	\item[\S\ref{PfThmSect}.] Proof of the main theorem
	\item[\S\ref{DirectProdSect}.] Review of direct products (proof of \cref{CayleyByY})
	\end{enumerate}

\begin{ack}
I am grateful to A.\,Hujdurović and Đ.\,Mitrović of the University of Primorska for informing me of Example~\ref{nonabelEg}, and for kindly giving me permission to include it in this paper.
\end{ack}

\section{A crucial lemma} \label{LemmaSect}

All graphs in this paper are undirected, with no multiple edges. Loops are allowed, but they are not necessary for any of the arguments, so readers are welcome to assume that all graphs are simple. Readers at the other extreme, who want to discuss multiple edges (or edge-colourings), are referred to \cref{NotSimple} below, 
but these complications are forbidden in this \lcnamecref{LemmaSect}.

\begin{defn}[{\cite[p.~34]{GodsilRoyle}}] \label{CayleyDefn}
Let $S$ be a \emph{symmetric} subset of an abelian group~$G$. (This means that $-s \in S$, for all $s \in S$.) The corresponding \emph{Cayley graph} $\Cay(G; S)$ is the graph whose vertices are the elements of~$G$, and with an edge joining the vertices $g$ and~$h$ if and only if $g = s + h$ for some $s \in S$.
\end{defn}

We now state a simple observation that is probably already in the literature somewhere. (Although the same proof also applies to Cayley \emph{digraphs}, we state the result only for Cayley \emph{graphs}, because they are the topic of this note.)

\begin{lem} \label{Aut(kS)}
Let $\varphi$ be an automorphism of a Cayley graph\/ $\Cay(G; S)$, and let $k \in \ZZ^+$. If 
\noprelistbreak
	\begin{enumerate} \itemsep=\smallskipamount 
	\item $G$ is abelian, 
	and 
	\item \label{Aut(kS)-neq}
	$ks \neq kt$ for all $s,t \in S$, such that $s \neq t$, 
	\end{enumerate}
then $\varphi$ is an automorphism of\/ $\Cay(G; k S)$, where $k S = \{\, ks \mid s \in S \,\}$.
\end{lem}

\begin{proof}
Write $k = p_1 p_2 \cdots p_r$, where each $p_i$ is prime, and let $k_i = p_1 p_2 \cdots p_i$ for $0 \le i \le r$. We will prove by induction on~$i$ that $\varphi$ is an automorphism of $\Cay(G, k_i S)$. 
The base case is true by assumption, since $k_0 S  = 1 S = S$.

For $v,w \in G$, let $\#(v,w)$ be the number of walks of length~$p_i$ from~$v$ to~$w$ in $\Cay(G, k_{i-1} S)$. These walks are in one-to-one correspondence with the $p_i$-tuples $(s_1,s_2,\ldots,s_{p_i})$ of elements of~$k_{i-1} S$, such that $s_1 + s_2 + \cdots + s_{p_i} = w - v$. Since $G$ is abelian, any cyclic rotation of $(s_1,s_2,\ldots,s_{p_i})$ also corresponds to a walk from~$v$ to~$w$. Therefore, the set of these walks can be partitioned into sets of cardinality~$p_i$, unless $w = p_i s + v$, for some $s \in k_{i-1} S$, in which case there is a walk of the form 
	$ v, s + v, 2s + v, \ldots, p_i s + v = w$.
(Also note that $s$ is unique, if it exists, by assumption~\pref{Aut(kS)-neq}.) Hence, we see that 
	\[ \#(v,w) \not\equiv 0 \pmod{p_i} \quad \iff \quad \text{$v$ is adjacent to~$w$ in $\Cay( G; p_i k_{i-1} S)$} . \]
Since $p_i k_{i-1} = k_i$, the desired conclusion  that $\varphi \in \Aut \Cay(G, k_i S)$ now follows from the induction hypothesis that $\varphi \in \Aut \Cay(G, k_{i-1} S)$ (and the observation that automorphisms preserve the value of the function~$\#$).
\end{proof}

\section{Comments on the lemma} \label{CommentsSect}

This \lcnamecref{CommentsSect} is optional.

\begin{rems}
Two comments on assumption~\pref{Aut(kS)-neq} of \cref{Aut(kS)}:
\noprelistbreak
\begin{enumerate}
	\item This assumption holds for all $S \subseteq G$ if and only if $\gcd \bigl( k, |G| \bigr) = 1$.
	\item This assumption can be weakened. For example, if $k$ is prime, then it suffices to assume, for each $s \in S$, that $|\{\, t \in S \mid ks = kt \,\}|$ is not divisible by~$k$.
	\end{enumerate}
\end{rems}

\begin{rem}
\Cref{Aut(kS)} may be of independent interest.  For example, it provides a short, fairly elementary proof of the known classification of edge-transitive graphs of prime order. (See the following \lcnamecref{EdgeTransitive}.) Can it also simplify the proofs of other known results \cite{JMorrisSurvey} on automorphism groups of circulant graphs?

\begin{cor}[Chao \cite{Chao}] \label{EdgeTransitive}
Let $\ZZ_p$ be the cyclic group of order~$p$, where $p$ is prime. A connected Cayley graph $X = \Cay(\ZZ_p; S)$ on~$\ZZ_p$ is edge-transitive if and only if $S$ is a coset of a subgroup of the multiplicative group~$\ZZ_p^\times$.
\end{cor}

\begin{proof}
($\Rightarrow$) 
First, note that $0 \notin S$, because a connected, edge-transitive graph with $p$ vertices cannot have loops. Since $p$ is prime, this implies $S \subseteq \ZZ_p^\times$.

Now, let $A_0 = \{\, \varphi \in \Aut X \mid \varphi(0) = 0 \,\}$ be the stabilizer of the vertex~$0$ in $\Aut X$. For every $k \in \ZZ_p^\times$, \cref{Aut(kS)} tells us that $k S$ is $A_0$-invariant. Since $S = 1S$ is also $A_0$-invariant, this implies that $kS \cap S$ is $A_0$-invariant. 
However, $A_0$ is transitive on~$S$ (because $S = N_X(0)$ and $X$ is edge-transitive), so this implies that either $kS = S$ or $k S = \emptyset$. Since this is true for all~$k$ (and $S \subseteq \ZZ_p^\times$), this means that $S$ is a \emph{block of imprimitivity} for the regular representation of~$\ZZ_p^\times$ \cite[p.~12]{DixonMortimer}. So $S$ is an orbit of some subgroup~$H$ of~$\ZZ_p^\times$ \cite[Thm.~1.5A, pp.~13--14]{DixonMortimer}: $S = Hz$ for some $z \in \ZZ_p^\times$. Since we are dealing with the regular representation, this means that $S$ is a coset of~$H$.

($\Leftarrow$) This is the easy direction (and does not require the assumption that $p$ is prime).
Assume $S$ is a coset of the subgroup~$H$ of~$\ZZ_p^\times$.
Note that $H$ acts by automorphisms on the group~$\ZZ_p$ (because multiplication by any $k \in \ZZ_p^\times$ is an automorphism of~$\ZZ_p$). The set~$S$ is invariant under~$H$; indeed, $H$ is transitive on~$S$ (because $S$ is a coset of~$H$). The proof is now completed by a well-known, elementary argument \cite[Lem.~2.8]{QinXiaZhou}: $H$ is a group of automorphisms of the Cayley graph~$X$ (because it is a group of automorphisms of~$\ZZ_p$ that fixes~$S$). Since $H$ fixes the vertex~$0$, and acts transitively on the set~$S$ of neighbours of~$0$, this implies that $X$ is edge-transitive.
\end{proof}
\end{rem}

\begin{rems} \label{NotSimple} 
Unlike in \cite{FernandezHujdurovic,QinXiaZhou}, we do not need to assume that $\Cay(G;S)$ is a simple graph.%
\noprelistbreak
	\begin{enumerate} \itemsep=\smallskipamount
	\item Graphs may have loops.
	\item \label{NotSimple-colour}
	We can allow edge-colourings of $\Cay(G;S)$ that are invariant under translation by elements of~$G$: the colour of an edge $(v,w)$ must be the same as the colour of the edge $(v + g, w + g)$. Such colourings come from colourings of~$S$: choose a colour for each element~$s$ of~$S$ (such that $\colour(s) = \colour(-s)$, for all $s \in S$), and then apply this colour to each edge of the form $(v, s + v)$.
Automorphisms are required to preserve the edge-colouring. 
	\item \Cref{Aut(kS)} remains valid this setting. To see this, let 
	\[ \text{$S_c = \{\, s \in S \mid \colour(s) = c \,\}$, \ for each colour~$c$.} \]
\Cref{Aut(kS)} (as stated, without considering any edge-colourings) implies that $\varphi$ is an automorphism of $\Cay(G; k S_c)$. Saying that this is true for every~$c$ is exactly the same as saying that $\varphi$ is a (colour-preserving) automorphism of $\Cay(G; kS)$, if we let 
	\[ \text{$\colour(t) = \{\, \colour(s) \mid t = ks, \ s \in S \,\}$ \  for each $t \in kS$.} \]
Also note that this proof only applies the original version of \cref{Aut(kS)} to~$S_c$, not all of~$S$, so hypothesis~\pref{Aut(kS)-neq} can be replaced with the weaker assumption that:
	\[ \text{(\ref{Aut(kS)-neq}$'$) \it for every colour~$c$, we have $ks \neq kt$ for all $s,t \in S_c$, such that $s \neq t$.} \]
	(If $s$ and~$t$ have different colours, then it is not necessary to assume $ks \neq kt$.)
	\item \label{NotSimple-multigraph}
	Technically, we do not allow graphs to have multiple edges. However, since the statements of the results only consider automorphism groups, not other graphical properties, the multiplicity of an edge can be encoded as part of its colour (or ``label''). For example, an edge coloured ``$2\mathsf{B}, 3\mathsf{R}, \mathsf{W}$'' could be thought of as representing 2~blue edges, 3~red edges, and a white edge, all with the same endpoints.
	\end{enumerate}
\end{rems}

\section{Proof of the main theorem} \label{PfThmSect}

\begin{defn}[Kotlov-Lovász {\cite{KotlovLovasz-twinfree}}] \label{TwinFreeDefn} 
A graph~$X$ is \emph{twin-free} if there do not exist two distinct vertices $v$ and~$w$, such that $N_X(v) = N_X(w)$, where $N_X(v)$ denotes the set of neighbours of~$v$ in~$X$.
\end{defn}

\begin{rem}
Synonyms for ``twin-free'' include ``irreducible'' \cite{FernandezHujdurovic}, ``$R$-thin'' \cite[p.~91]{Handbook}, and ``vertex-determining'' \cite{QinXiaZhou}.
\end{rem}

Let $S$ be a symmetric subset of a finite abelian group~$G$ of odd order, such that the Cayley graph $X = \Cay(G;S)$ is twin-free and connected. 
Given $\varphi \in \Aut BX$, we wish to show that $\varphi \in \Aut X \times S_2$.

Note that 	
	\[ BX = \Cay \bigl( G \times \ZZ_2 ; S \times \{1\} \bigr) , \]
and that $BX$ is connected and bipartite, with bipartition sets $G \times \{0\}$ and $G \times \{1\}$.
Since $\Aut X \times S_2$ contains an element that interchanges these two sets, we may assume
	\[ \text{$\varphi\bigl( G \times \{i\} \bigr)  = G \times \{i\}$ for $i = 0,1$.} \]

\begin{rem}
If the edges of~$X$ have been coloured (as allowed by \fullcref{NotSimple}{colour}):
\noprelistbreak
	\begin{enumerate} \itemsep=\smallskipamount
	\item Each element~$x$ of $N_X(v)$ is labelled with the colour of the edge from~$v$ to~$x$. Therefore, saying that $N_X(v) = N_X(w)$ means that, for each $x \in N_X(v)$, the colour of the edge joining $x$ to~$v$ is same as the colour of the edge joining $x$ to~$w$. Hence, an edge-coloured graph may be twin-free, even though its underlying uncoloured graph is not twin-free.
	\item The edges of $BX$ are coloured by colouring the edge from $(v,0)$ to $(w,1)$ with whatever colour appears on the edge from $v$ to~$w$ in~$X$.
	\end{enumerate}
\end{rem}

Let $k = |G| + 1$, so $k \equiv 1 \pmod{|G|}$ and $k \equiv 0 \pmod{2}$.
Then, for every $(s,1) \in S \times \{1\}$, we have $k (s,1) = (s,0)$, so \cref{Aut(kS)} tells us that $\varphi$ is an automorphism of the graph $\Cay \bigl( G \times \ZZ_2; S \times \{0\} \bigr)$ (which is the disjoint union of two copies of~$X$). Hence, after multiplying by an element of $\Aut X \times \{\iota\}$, where $\iota$~is the identity element of~$S_2$, we may assume that $\varphi(v) = v$ for all $v \in G \times \{0\}$.

It is now easy to complete the proof, by using the assumption that $X$ is twin-free. For all $g \in G$, we have
	\begin{align*}
	N_X(g) \times \{0\}
	&= \varphi \bigl( N_X(g) \times \{0\} \bigr)
		&& \text{($\varphi(v) = v$ for all $v \in G \times \{0\}$)}
	\\&= \varphi \bigl( N_{BX}(g,1) \bigr)
		&& \text{(definition of $BX$)}
	\\&= N_{BX}\bigl( \varphi(g,1) \bigr)
		&& \text{($\varphi$ is an automorphism)}
	\\&= N_{BX} ( g', 1)
		&& \text{(where $\varphi(g,1) = (g', 1)$)}
	\\&= N_X(g') \times \{0\}
		&& \text{(definition of $BX$)}
	. \end{align*}
Since $X$ is twin-free, this implies $g = g'$, so $\varphi(g,1) = (g',1) = (g,1)$, which means $\varphi(v) = v$ for all $v \in G \times \{1\}$. Since this equality also holds for all $v \in G \times \{0\}$, we conclude that $\varphi$ is the identity element of $\Aut BX$, and is therefore in the subgroup $\Aut X \times S_2$. 
\qed

\section{Review of direct products} \label{DirectProdSect}

\begin{defn}[{}{\cite[p.~36]{Handbook}}] \label{DirectProdDefn}
The \emph{direct product} of two graphs $X$ and~$Y$ is the graph $X \times Y$ with $V(X \times Y) = V(X) \times V(Y)$, and 
	\[ \text{$(x_1,y_1)$ is adjacent to $(x_2,y_2)$ in $X \times Y$}
	\quad \iff \quad
	\begin{matrix} \text{$x_1$ is adjacent to~$x_2$ in~$X$, and} \hfill \\ \text{$y_1$ is adjacent to~$y_2$ in~$Y$} . \end{matrix} \]
\end{defn}

\begin{rem}[{}{\cite[p.~36]{Handbook}}]
The literature has numerous other names for the direct product, including ``tensor product\rlap,'' ``Kronecker product\rlap,'' ``cardinal product\rlap,'' and ``conjunction\rlap.''
\end{rem}

The graph~$X$ in \cref{CayleyByK2} is allowed to have edge-colours and multiple edges, but the theory of automorphisms of direct products does not seem to have been developed in this generality, so:

\begin{assumps}
In this \lcnamecref{DirectProdSect} (and, therefore, in \cref{CayleyByY}), graphs do \emph{not} have edge-colours or multiple edges (but they may have loops). 
\end{assumps}

As was mentioned in \cref{IntroSect}, we have $BX = X \times K_2$.
Generalizing the comments there about $\Aut BX$, it is clear that $\Aut X \times \Aut Y$ is a subgroup of $\Aut(X \times Y)$, and that if this subgroup happens to be all of $\Aut(X \times Y)$, then $X$ and~$Y$ must be connected, and must also be twin-free.
(We ignore the situation where one of the graphs is the one-vertex trivial graph.)

For direct products of non-bipartite graphs, the converse holds if and only if a certain ``coprimality'' condition holds. However, instead of stating the full strength of this classical theorem of W.\,D\"orfler, we present only a simpler, weakened version of the result:

\begin{thm}[D\"orfler, cf.\ \cite{Dorfler} or {\cite[Thm.~8.18, p.~103]{Handbook}}] \label{AutNonbipDirProd}
Let $X$ and~$Y$ be twin-free, connected, non-bipartite graphs of relatively prime order. Then $\Aut(X \times Y) = \Aut X \times \Aut Y$.
\end{thm}

The situation is more complicated (and not yet understood) when one of the factors of a direct product is bipartite. However, the following facts shed some light, as will be seen in the \lcnamecref{AutBipDirProd} that follows.

\begin{facts} \label{DirFacts}
Let $X$ and~$Y$ be connected graphs.
\noprelistbreak
	\begin{enumerate} \itemsep = \smallskipamount
	\item The \emph{Cartesian skeleton} of~$X$ is a certain graph~$\skel X$ that is defined from~$X$, such that $V(\skel X) = V(X)$ \cite[Defn.~8.2, p.~95]{Handbook}.
	\item \label{DirFacts-natural}
	Every automorphism of $X \times Y$ is also an automorphism of $\skel (X \times Y)$ \cite[Prop.~8.11, p.~97]{Handbook}.
	\item \label{DirFacts-S(x)}
	If $X$ and~$Y$ are twin-free (and have more than one vertex), then $\skel (X \times Y) = \skel X \cartprod \skel Y$ \cite[Prop.~8.10, p.~96]{Handbook}, where $\cartprod$ denotes the \emph{Cartesian product} \cite[p.~35]{Handbook}.
	\item \label{DirFacts-Aut(cart)}
	If $|V(X)|$ is relatively prime to $|V(Y)|$, then $\Aut(X \cartprod Y) = \Aut X \times \Aut Y$ \cite[Cor.~6.12, p.~70]{Handbook}.
	\item \label{DirFacts-conn}
	If $X$ is not bipartite, then $\skel X$ is connected \cite[Prop.~8.13(i), p.~98]{Handbook}.
	\item \label{DirFacts-notconn}
	If $Y$ is bipartite, then $\skel Y$ has precisely two connected components, and their vertex sets are the bipartition sets of~$Y$ \cite[Prop.~8.13(ii), p.~98]{Handbook}.
	\end{enumerate}

\end{facts}

The following straightforward consequence of these facts is presumably known to experts, but we do not have a reference. 
The gist is that, in order to understand the automorphism group of $X \times Y$, where $Y$ is bipartite, it often suffices to understand the special case where $Y = K_2$. 

\begin{prop} \label{AutBipDirProd}
Let $X$ and~$Y$ be twin-free, connected graphs that have at least one edge, such that:
\noprelistbreak
	\begin{enumerate} \itemsep = \smallskipamount
	\item $X$ is not bipartite, 
	\item $Y$ is bipartite, with bipartition $V(Y) = Y_0 \cup Y_1$, such that
		\noprelistbreak
		\begin{enumerate} \itemsep = \smallskipamount
		\item  \label{AutBipDirProd-bip-relprime}
			$|Y_0|$ and\/ $|Y_1|$ are relatively prime to $|V(X)|$,
			and
		\item \label{AutBipDirProd-bip-interchange}
			either
		 	$Y$ has an automorphism that interchanges $Y_0$ and~$Y_1$,
			or\/
			$|Y_0| \neq |Y_1|$,
	and 
		\end{enumerate}
	\item \label{AutBipDirProd-AutBx}
	$\Aut BX = \Aut X \times S_2$.
	\end{enumerate}
Then $\Aut(X \times Y) = \Aut X \times \Aut Y$.
\end{prop}

\begin{proof}
Let $\varphi \in \Aut(X \times Y)$. From Fact~\ref{DirFacts-notconn}, we know that $\skel Y$ has two connected components~$C_0$ and~$C_1$, where $V(C_0) = Y_0$ and $V(C_1) = Y_1$. 
By Facts~\ref{DirFacts-natural} and~\ref{DirFacts-S(x)}, we have 
	\[ \varphi \in \Aut \bigl( \skel (X \times Y) \bigr) = \Aut(\skel X \cartprod \skel Y) = \Aut(\skel X \cartprod (C_0 \cup C_1) \bigr) . \]
If $|V(C_0)| = |V(C_1)|$, then (by hypothesis~\pref{AutBipDirProd-bip-interchange}) $Y$ has an automorphism that interchanges $C_1$ and~$C_2$; therefore, we may assume that $\varphi$ fixes each connected component of $\skel X \cartprod \skel Y$. This means that $\varphi$ restricts to an automorphism~$\varphi_i$ of $\skel X \cartprod C_i$ (for $i = 0,1$). Since $\skel X$ is connected (by Fact~\ref{DirFacts-conn}), and the connected component~$C_i$ is obviously also connected, Fact~\ref{DirFacts-Aut(cart)} (and hypothesis~\pref{AutBipDirProd-bip-relprime}) tells us that there exist permutations $\chi_i$ of $V(X)$ and $\eta_i$ of~$Y_i$, such that 
	\[ \text{$\varphi(x,y) = \bigl( \chi_i(x), \eta_i(y) \bigr)$ for all $x \in V(X)$ and $y \in Y_i$} . \]

Choose some edge $(y_0,y_1)$ of~$Y$, with $y_i \in Y_i$, and let $y_i' = \eta_i(y_i)$. Let $B$ and~$B'$ be the subgraphs of $X \times Y$ induced by $V(X) \times \{y_0,y_1\}$ and $V(X) \times \{y_0', y_1' \}$, respectively, so $B' = \varphi(B)$.
By definition of the direct product, the maps $(x,i) \mapsto (x, y_i)$ and $(x,i) \mapsto (x, y_i')$ are isomorphisms from~$BX$ to $B$ and~$B'$. Therefore, the map $(x,i) \mapsto \bigl(\chi_i(x), i \bigr)$ is an automorphism of~$BX$. So Assumption~\pref{AutBipDirProd-AutBx} tells us that $\chi_0 = \chi_1$.
This means that the $V(X)$-component of $\varphi(x,y)$ depends only on~$x$. (And we already knew that the $V(Y)$-component of $\varphi(x,y)$ depends only on~$y$.) This easily implies $\varphi \in \Aut X \times \Aut Y$.
\end{proof}

\begin{proof}[\bf Proof of \cref{CayleyByY}]
Combine \cref{AutNonbipDirProd,AutBipDirProd} (and note that \cref{CayleyByK2} verifies hypothesis~\fullref{AutBipDirProd}{AutBx}).
\end{proof}


\begin{thebibliography}{[00]}

\larger
\itemsep=\medskipamount

\bibitem{BeckenbachBorndorfer-stable}
I.\,Beckenbach and R.\, Borndörfer:
Hall's and Kőnig's theorem in graphs and hypergraphs. 
\emph{Discrete Math.} 341 (2018), no.~10, 2753--2761. 
\MR{3843262},
\href{https://doi.org/10.1016/j.disc.2018.06.013}{doi:10.1016/j.disc.2018.06.013}

\bibitem{Chao}
C.-Y.\,Chao:
On the classification of symmetric graphs with a prime number of vertices.
\emph{Trans. Amer. Math. Soc.} 158 (1971) 247--256.
\MR{0279000},
\href{https://doi.org/10.1090/S0002-9947-1971-0279000-7}{doi:10.1090/S0002-9947-1971-0279000-7}

\bibitem{DixonMortimer}
J.\,D.\,Dixon and B.\,Mortimer:
\emph{Permutation Groups}.
Springer, New York, 1996.
\MR{1409812},
\href{https://doi.org/10.1007/978-1-4612-0731-3}{doi:10.1007/978-1-4612-0731-3}

\bibitem{Dorfler}
W.\,Dörfler:
Primfaktorzerlegung und Automorphismen des Kardinalproduktes von Graphen.
\emph{Glasnik Mat. Ser. III} 9 (29) (1974) 15--27. 
\MR{0387113},
\url{http://books.google.com/books?id=tuAwMRMivX4C&pg=PA15}

\bibitem{FernandezHujdurovic}
B.\,Fernandez and A.\,Hujdurović:
Canonical double covers of circulants
(preprint, 2020).
\url{https://arxiv.org/abs/2006.12826}

\bibitem{GoddardWinter-stable}
W.\,Goddard and P.\,A.\,Winter:
A characterization of stable graphs on a maximum (minimum) number of edges.
\emph{Quaestiones Math.} 10 (1986), no.~2, 175--178. 
\MR{0871018},
\href{https://doi.org/10.1080/16073606.1986.9631602}{doi:10.1080/16073606.1986.9631602}

\bibitem{GodsilRoyle}
C.\,Godsil and G.\,Royle:
\emph{Algebraic Graph Theory.} 
Springer, New York, 2001.
\MR{1829620},
\href{https://doi.org/10.1007/978-1-4613-0163-9}{doi:10.1007/978-1-4613-0163-9}

\bibitem{Handbook}
R.\,Hammack, W.\,Imrich, and S.\,Klavžar:
\emph{Handbook of Product Graphs, 2nd ed.}
CRC Press, Boca Raton, FL, 2011.
\MR{2817074},
\url{https://www.routledge.com/9781138199088}

\bibitem{KotlovLovasz-twinfree}
A.\,Kotlov and L.\,Lovász:
The rank and size of graphs.
\emph{J. Graph Theory} 23 (1996), no.~2, 185--189. 
\MR{1408346},
\href{https://doi.org/10.1002/(SICI)1097-0118(199610)23:2<185::AID-JGT9>3.0.CO;2-P}{doi:10.1002/(SICI)1097-0118(199610)23:2<185::AID-JGT9>3.0.CO;2-P}

\bibitem{JMorrisSurvey}
J.\,Morris:
Automorphism groups of circulant graphs---a survey, 
in A.\,Bondy, et al., eds., \emph{Graph Theory in Paris}.
Birkh\"auser, Basel, 2007, pp.~311--325. 
\MR{2279185},
\href{https://doi.org/10.1007/978-3-7643-7400-6_24}{doi:10.1007/978-3-7643-7400-6\underline{\phantom{x}}24}

\bibitem{QinXiaZhou}
Y.-L.\,Qin, B.\,Xia, and S.\,Zhou:
Stability of circulant graphs,
\emph{J. Combin. Theory Ser.~B} 136 (2019) 154--169. 
\MR{3926283}, 
\href{https://doi.org/10.1016/j.jctb.2018.10.004}{doi:10.1016/j.jctb.2018.10.004}

\bibitem{Sabidussi}
G.\,Sabidussi:
The composition of graphs,
\emph{Duke Math. J.} 26 (1959) 693--696. 
\MR{0110649},
\href{https://doi.org/10.1215/S0012-7094-59-02667-5}{doi:10.1215/S0012-7094-59-02667-5}

\bibitem{CanCover}
Wikipedia: Bipartite double cover.
\url{https://en.wikipedia.org/wiki/Bipartite_double_cover}

\end{thebibliography}
\end{document}